\documentclass{amsart}
\usepackage{amssymb,stmaryrd,mathrsfs,dsfont,amsmath}
\usepackage{colonequals}

\theoremstyle{plain}
\newtheorem{theorem}{Theorem}[section]
\newtheorem{lemma}[theorem]{Lemma}
\newtheorem{proposition}[theorem]{Proposition}
\newtheorem{corollary}[theorem]{Corollary}

\theoremstyle{definition}
\newtheorem{notation}[theorem]{Notation}
\newtheorem{example}[theorem]{Example}

\DeclareMathOperator{\N}{\mathbb{N}}
\DeclareMathOperator{\Z}{\mathbb{Z}}

\input xy

\usepackage{microtype}
\begin{document}
	

\title[On Isolated Gaps of Numerical Semigroups]{On Isolated Gaps of Numerical Semigroups of embedding dimension two}

\author[Shubh N. Singh]{\bfseries Shubh N. Singh}
\address{Department of Mathematics, Central University of South Bihar, Gaya--824236, Bihar, India}
\email{shubh@cub.ac.in}

	\author[Ranjan K. Ram]{\bfseries Ranjan K. Ram}
	\address{Department of Mathematics, Central University of South Bihar, Gaya--824236, Bihar, India}
	\email{ranjankr@cusb.ac.in}


	
\begin{abstract}


We explicitly describe all the isolated gaps of any numerical semigroup of embedding dimension two, and we give an exact formula for the number of isolated gaps of these numerical semigroups.

\end{abstract}
\maketitle

\section{Introduction}
Let $\N$ denote the set of all nonnegative integers. A \emph{numerical semigroup} is a nonempty subset $S$ of $\N$ such that $S$ is closed under addition, $0\in S$, and the set $\N\setminus S\colonequals \{n\in \mathbb{N}\colon n\notin S\}$ is finite. The \emph{set of gaps} $G(S)$ of a numerical semigroup $S$ is the finite set $\N\setminus S$. Given positive integers $a_1< \cdots< a_m$, let $\langle a_1, \ldots, a_m\rangle \colonequals \{\lambda_1 a_1 + \cdots + \lambda_m a_m\colon \lambda_1, \ldots, \lambda_m \in \N\}$. It is well-known that $\langle a_1, \ldots, a_m\rangle$ is a numerical semigroup if and only if the greatest common divisor of $a_1, \ldots, a_m$ is one (cf. \cite[Lemma 2.1]{rosales-b09}), and every numerical semigroup has this form (cf. \cite[Theorem 2.7]{rosales-b09}).

\vspace{0.5mm}

In $2019$, Moreno-Fr{\'i}as and Rosales \cite[p. 1742]{rosales_tjm19} defined an \emph{isolated gap} of a numerical semigroup $S$ to be an element $x\in G(S)$ such that $\{x-1, x+1\} \subseteq S$, and called a numerical semigroup \emph{perfect} if it does not contain any isolated gaps. Moreno-Fr{\'i}as and Rosales proved that any numerical semigroup of embedding dimension two is not perfect \cite[Corollary 4]{rosales-pmd20}. Further, Moreno-Fr{\'i}as and Rosales characterized perfect numerical semigroups of embedding dimension three \cite[Proposition 8]{rosales-pmd20}. Smith studied various properties of isolated gaps of any non-perfect numerical semigroup \cite{smith-tjm22}. Smith proved that if $h$ is the smallest isolated gap of a numerical semigroup $\langle a, b \rangle$, then the numerical semigroup $\langle a, b, h\rangle$ is perfect \cite[Theorem 3.1]{smith-tjm22}.

\vspace{0.5mm}
One of the main results of this paper is Theorem \ref{th_card-of-iso-gaps-set}, which determines the number of isolated gaps of $S$. Theorem \ref{th_card-of-iso-gaps-set} permit us to explicitly describe all the isolated gaps of these numerical semigroups.

\vspace{0.5mm}
The rest of the paper is organized as follows: In Section $2$, we recall preliminary notions, introduce notation, and state some well known results. In Section 3, we determine the number of isolated gaps of any numerical semigroup of embedding dimension two. In Section 4, we describe all the isolated gaps of any numerical semigroup of embedding dimension two. We further give an exact formula for the smallest isolated gap of these numerical semigroups.


\section{Preliminaries and Notation}

Let $X$ be a nonempty set. We denote by $\# X$ the cardinality of $X$. A \emph{partition} of $X$ is a collection of pairwise disjoint nonempty subsets, called \emph{blocks}, whose union is $X$. A partition is \emph{uniform} if all its blocks have the same cardinality. Let $f\colon X \to Y$ be a function. We write $f(x)$ to denote the image of an element $x\in X$ under $f$. For a subset $A$ of $X$ and a subset $B$ of $Y$, let $f(A)\colonequals \{f(x)\colon x\in A\}$ and $f^{-1}(B)\colonequals \{x\in X\colon  f(x) \in B\}$. The absolute value of a real number $x$ is denoted by $|x|$.

\vspace{0.5mm}
The set of all integers is denoted by $\Z$. For an integer $m\ge 2$, the greatest common divisor of the positive integers $a_1, \ldots, a_m$ is denoted by $\gcd(a_1, \ldots, a_m)$. For integers $x$ and $y$ with $x<y$, let $[x, y] \colonequals \{z\in \Z\colon x\le z\le y\}$. For an integer $k$ and a set $A$ of integers, let $kA\colonequals \{ka\colon a\in A\}$.  We denote by $\min A$ (resp. $\max A$) the smallest (resp. greatest) element of a finite nonempty set $A$ of integers. For integers $a$, $b$, $m$ with $m\ge 1$, we say that \emph{$a$ is congruent to $b$ modulo $m$}, written $a\equiv b \pmod m$, if $m$ divides $a-b$. A \emph{complete residue system modulo $m$} is a set of $m$ integers such that every integer is congruent modulo $m$ to exactly one element of the set. Let $a$ and $b$ be relatively prime positive integers. A solution $(u, v)$ of the Diophantine equation $ax+by = 1$ is a \emph{definitely least solution} (\emph{d.l.s.}) if both $|u|$ and $|v|$ attain their least possible values (cf. \cite[p. 646]{levit-amm56}). It is known that there is a unique d.l.s. of the Diophantine equation $ax+by = 1$ \cite[Theorem 3]{levit-amm56}.

\vspace{0.5mm}
Let $S$ be a numerical semigroup. We denote by $I(S)$ the set of all isolated gaps of $S$. We use $F(S)$ to denote the Frobenius number of $S$. Let $N(S)\colonequals \{s\in S\colonequals s< F(S)\}$.  It is well-known that every numerical semigroup admits a unique minimal (under inclusion) finite generating set (cf. \cite[Theorem 2.7]{rosales-b09}). The \emph{embedding dimension} of $S$ is the cardinality of the minimal generating set of $S$. We say that $S$ is \emph{symmetric} if $F(S)-x \in S$ for all $x\in \Z\setminus S$. It is well-known that every numerical semigroup of embedding dimension two is symmetric (cf. \cite[Corollary 4.7]{rosales-b09}). For an element $n\in S\setminus \{0\}$, the \emph{Ap\'ery set} of $n$ in $S$ is $\text{Ap}(S, n) = \{s\in S\colon s-n\notin S\}$.

\vspace{0.5mm}
We refer the reader to \cite{rosales-b09} for any undefined notions and results on numerical semigroups. In the rest of the paper, we will henceforth assume that $a$ and $b$ are relatively prime positive integers such that $1<a<b$, and $(u,v)$ is the d.l.s. of the equation $ax+by = 1$. 

\vspace{0.5mm}

We end this section by listing two known lemmas.
\begin{lemma}\cite[Lemma]{levit-amm56}\label{le_iff-condition-for-dls}
A solution $(x_0, y_0)$ of the Diophantine equation $ax+by = 1$ is the d.l.s. if and only if $|x_0|\le \frac{b}{2}$ and $|y_0| \le \frac{a}{2}$.
\end{lemma}

\begin{lemma}\cite[Lemma 1]{rosales-laa05}\label{le_rosales-laa05-le1}
Let $x\in \Z$. Then $x\notin \langle a, b\rangle$ if and only if $x = ab-\ell a-kb$ for some $\ell, k\in \N\setminus \{0\}$.
\end{lemma}


\section{Combinatorial results}


Our first aim in this section is to prove Theorem \ref{th_card-of-iso-gaps-set}, which determines an exact formula for the number of isolated gaps of a numerical semigroup of embedding dimension two.
 
\vspace{0.2mm}
Given a numerical semigroup $S$, we define \[T(S) \colonequals \{s\in N(S) \colon  s-1,s+1 \notin S\}.\]

\vspace{0.5mm}
The following proposition shows that the sets $I(S)$ and $T(S)$ have the same cardinality if the numerical semigroup $S$ is symmetric.

\begin{proposition}\label{pr_card-of-IS-NS-are-same}
If $S$ is a symmetric numerical semigroup, then $\# I(S)= \# T(S)$.
\end{proposition}

\begin{proof}[{\bf Proof}]
Define a function $\alpha\colon T(S) \to I(S)$ by $\alpha(x)= F(S) -x$. First, we show that $\alpha(x)\in I(S)$ whenever $x\in T(S)$. Let $x\in T(S)$. Then $\alpha(x) = F(S)-x\in G(S)$. Now since $x-1, x+1\notin S$ and $S$ is symmetric, we obtain $F(S)- (x-1), F(S)-(x+1)\in S$. Thus $F(S)-x\in I(S)$, and so $\alpha(x)\in I(S)$.

\vspace{0.2mm}
It is clear from the definition that $\alpha$ is injective. Finally, we show that $\alpha$ is surjective. Let $y\in I(S)$. Since $S$ is symmetric, it follows that $F(S)- y \in N(S)$. Now since $y-1, y+1\in S$, we obtain $F(S)- (y-1), F(S)- (y+1)\notin S$. Thus $F(S)-y\in T(S)$. Moreover, observe that $\alpha\left(F(S)-y\right) = y$, as desired.
\end{proof}


\begin{notation}
Throughout the rest of the paper, we will use the symbol $\alpha$ to denote the bijection defined as in the proof of Proposition \ref{pr_card-of-IS-NS-are-same}.
\end{notation}


The following lemma determines the cardinality of the finite set $T(\langle a, b \rangle)$.
\begin{lemma}\label{le_card-of-the-subset-of-NS}
If $S= \langle a,b\rangle$, then $\# T(S) = |uv|$.
\end{lemma}

\begin{proof}[{\bf Proof}]
Let $s \in T(S)$. Then $s = a\lambda_1+b\lambda_2$ for some $\lambda_1,\lambda_2 \in \mathbb{N}$. Since $1=au+bv$, we deduce that $s-1 = a(\lambda_1-u) +b(\lambda_2-v)$ and $s+1 = a(\lambda_1+u) +b(\lambda_2+v)$. Note that exactly one of the integers $u$ and $v$ is negative.
	
	
\vspace{0.2mm}
\noindent {Case 1:} Suppose $u <0$. Write $u'=-u$. Then $s-1 = a(\lambda_1+u') +b(\lambda_2-v)$ and $s+1 = a(\lambda_1-u') +b(\lambda_2+v)$. Since $s-1, s+1 \notin S$, it follows that $\lambda_2 <v$ and $\lambda_1 <u'$. 

We now show that if $\lambda_1 <u'$ and $\lambda_2 <v$, then $s-1, s+1 \notin S$. Since $u'\leq \frac{b}{2}$ and $v \leq \frac{a}{2}$ by Lemma \ref{le_iff-condition-for-dls}, we obtain $\lambda_1+u' < 2u' \leq b$ and $\lambda_2+v < 2v \leq a$. Put $\ell=a-(\lambda_2+v)$, $k=b-(\lambda_1+u')$, $r=u'-\lambda_1$, and $t=v-\lambda_2$. Then
$s-1 =a(b-k) + b(-t)= ab -ak-bt$ and $s+1= a(-r) + b(a-\ell) = ab- ar -b\ell$.
Since $\ell,k,r,t\in \N\setminus \{0\}$, we obtain $s-1, s+1 \notin S$ by Lemma \ref{le_rosales-laa05-le1}.

\vspace{0.2mm}
	\noindent {Case 2:} Suppose $v<0$. Write $v'=-v$. Then $s-1 = a(\lambda_1-u)+ b(\lambda_2+v')$ and $s+1 = a(\lambda_1+u)+ b(\lambda_2-v')$. Since $s-1, s+1 \notin S$, it follows that $\lambda_1 <u$ and $\lambda_2 <v'$. By using the similar argument as in Case 1, we can readily show that if $\lambda_1 <u$ and $\lambda_2 <v'$, then $s-1, s+1 \notin S$.

\vspace{0.2mm}
In each case, we have $s-1, s+1 \notin S$ whenever $\lambda_1<|u|$ and $\lambda_2 <|v|$. Hence we conclude that $\# T(S) = |u||v| = |uv|$.
\end{proof}

Recall that every numerical semigroup of embedding dimension two is symmetric. By combining Proposition \ref{pr_card-of-IS-NS-are-same} and Lemma \ref{le_card-of-the-subset-of-NS}, we thus obtain the following theorem.

\begin{theorem}\label{th_card-of-iso-gaps-set}
If $S= \langle a,b\rangle$, then $\# I(S)=|uv|$. 
\end{theorem}

We illustrate Theorem \ref{th_card-of-iso-gaps-set} with the next two examples.
\begin{example}\label{ex-1_card-of-IS}
Consider the numerical semigroup $S=\langle 9,13 \rangle$. Observe that $(3, -2)$ is a solution of the equation $9x+13y = 1$. Moreover, the solution $(3, -2)$ is the d.l.s. by Lemma \ref{le_iff-condition-for-dls}. Therefore $u = 3$ and $v = -2$. By using Theorem \ref{th_card-of-iso-gaps-set}, we thus obtain $\# I(S) = |uv|= 6$. Indeed, a modest computation confirms that $I(S) =\{64,73,77,82,86,95\}$.
\end{example}

\begin{example}
Consider the numerical semigroup $S=\langle 8,13 \rangle$. Observe that $(5, -3)$ is a solution of the equation $9x+13y = 1$. Moreover, the solution $(5, -3)$ is the d.l.s. by Lemma \ref{le_iff-condition-for-dls}. Therefore $u = 5$ and $v = -3$. By using Theorem \ref{th_card-of-iso-gaps-set}, we thus obtain $\# I(S) = |uv|= 15$. Indeed, a modest computation confirms that $I(S) =\{25, 33, 38, 41, 46, 49, 51, 54, 57, 59, 62, 67, 70, 75, 83\}$.
\end{example}


Let $S$ be a non-perfect numerical semigroup, and let $m\in S\setminus \{0\}$. For each $i\in [1, m-1]$, Smith \cite[p. 124]{smith-tjm22} defined the set $I_{i,m}(S) \colonequals \{s\in I(S)\colon s \equiv i \pmod m\}$. It is easy to verify that the collection $\left\{I_{i,m}(S)\colon I_{i,m}(S) \neq \varnothing,\; i\in [1, m-1]\right\}$ is a partition of $I(S)$. Smith \cite[Theorem 2.7]{smith-tjm22} proved that the partition \newline
$\left\{I_{i,a}(\langle a,b \rangle)\colon I_{i,a}(\langle a,b \rangle) \neq \varnothing,\; i\in [1, a-1]\right\}$ of $I(\langle a,b \rangle)$ is uniform.

\vspace{0.5mm}
To determine the cardinality of $\left\{I_{i,a}(\langle a,b \rangle)\colon I_{i,a}(\langle a,b \rangle) \neq \varnothing,\; i\in [1, a-1]\right\}$, we require the next two lemmas.

\begin{lemma}\label{le_partition-of-TS}
Let $S$ be a symmetric numerical semigroup such that $I(S)\neq \varnothing$, and let $m\in S\setminus \{0\}$. Then:
\begin{enumerate}
\item[\rm(i)] $\left\{\alpha^{-1}(I_{i,m}(S))\colon I_i(S)\neq \varnothing,\; i\in [1, m-1] \right\}$ is a partition of $T(S)$,

\item[\rm(ii)] any two elements of $\alpha^{-1}(I_{i,m}(S))$ are congruent modulo $m$.
\end{enumerate}
\end{lemma}

\begin{proof}[{\bf Proof}]
Note that $\alpha\colon T(S)\to I(S)$ is a bijection. 

\begin{enumerate}
 \item[\rm(i)] This follows from the facts that the inverse image of a finite union (resp. intersection) equals the union (resp. intersection) of the inverse images.

\item[\rm(ii)] Let $x, y\in\alpha^{-1}(I_{i,m}(S))$. Then $\alpha(x), \alpha(y)\in I_{i,m}(S)$, and so $\alpha(x)\equiv \alpha(y) \pmod m$. Then, by the definition of $\alpha$ we obtain $F(S)- x\equiv F(S)- y \pmod m$. This implies that $x\equiv y \pmod m$, as desired.
\end{enumerate}

\end{proof}


\begin{lemma}\label{le_no-of-cells-in TS}
If $S = \langle a,b \rangle$, then $\#\left\{\alpha^{-1}(I_{i,a}(S))\colon I_{i,a}(S)\neq \varnothing,\; i\in [1, a-1] \right\} = |v|$.
\end{lemma}

\begin{proof}[{\bf Proof}]
Note that $\text{Ap}(S,a) = \{0,b,\cdots,(a-1)b\}$ (cf. \cite[p. 10]{rosales-b09}). Moreover, the Ap\'ery set $\text{Ap}(S,a)$ is a complete 
residue system modulo $a$ (cf. \cite[Theorem 2.6]{niven-b91}). Thus, by Lemma \ref{le_partition-of-TS} it suffices to count elements in $T(S)$ that has the form $bk$ for some $k\in \{0, \ldots, a-1\}$.

	\vspace{0.2mm}
Let $bk\in T(S)$. Since $1 = au+bv$, we deduce that
	$bk+1 = au+b(k+v)$ and $bk-1 = -au+b(k-v)$. From Lemma \ref{le_iff-condition-for-dls}, we know that $|u|< b$ and $|v|< a$. Note that exactly one of $u$ and $v$ is negative.

	\vspace{0.2mm}
	\noindent {Case 1:} Suppose $u<0$. Write $u'=-u$ and $\ell= b-u'$. Then 
	$bk-1= a(b-\ell)+ b(k-v) = ab-a\ell -b(v-k)$. Since  $bk-1\notin S$ and $\ell \in \N\setminus \{0\}$, it follows that $k<v$ by Lemma \ref{le_rosales-laa05-le1}.
	
	We now show that if $k<v$, then $bk+1\notin S$. Put $t= a-(k+v)$. Observe that $k+v < 2v \le a$, and so $t\in \N\setminus \{0\}$.  Then
	$bk+1 = -au'+ b(a-t) = ab-au'-at$. Since $u', t\in \N\setminus \{0\}$, we obtain $bk+1\notin S$ by Lemma \ref{le_rosales-laa05-le1}.
	
	\vspace{0.2mm}
	\noindent {Case 2:} Suppose $v<0$. Write $v'=-v$ and $\ell = b-u$. Then 
	$bk+1 = a(b-\ell)+ b(k-v') = ab-a\ell -b(v'-k)$. Since  $bk+1\notin S$ and $\ell \in \N\setminus \{0\}$, we obtain $k<v'$ by Lemma \ref{le_rosales-laa05-le1}.
	
We now show that if $k<v'$, then $bk-1\notin S$. Put $t = a-(k+v')$. Observe that $k+v' < 2v' \le a$, and so $t\in \N\setminus \{0\}$. Then
	$bk-1 = -au+ b(a-t) = ab-au-at$. Since $u, t\in \N\setminus \{0\}$, we obtain $bk-1\notin S$ by Lemma \ref{le_rosales-laa05-le1}.

\vspace{0.2mm}
In each case, we have $k<|v|$. Hence $\#\left\{\alpha^{-1}(I_{i,a}(S))\colon I_{i,a}(S)\neq \varnothing,\; i\in [1, a-1] \right\} = |v|$.
\end{proof}


\begin{proposition}\label{pr_no-of-cells-in IS}
If $S = \langle a,b \rangle$, then $\# \left\{I_{i,a}(S)\colon I_{i,a}(S) \neq \varnothing,\; i\in [1, a-1]\right\} = |v|$.
\end{proposition}
\begin{proof}[{\bf Proof}]
This follows from Lemma \ref{le_partition-of-TS}\rm(i) and Lemma \ref{le_no-of-cells-in TS}, since $\alpha \colon T(S)\to I(S)$ is a bijection.
\end{proof}

We illustrate Proposition \ref{pr_no-of-cells-in IS} with the following example.

\begin{example}\label{ex_no-of-blocks}
Consider the numerical semigroup $S =\langle 7,16 \rangle$. Observe that $(7 -3)$ is a solution of the equation $7x+16 y = 1$. Moreover, the solution $(7, -3)$ is the d.l.s. by Lemma \ref{le_iff-condition-for-dls}. Therefore $|v| = 3$. By using Proposition \ref{pr_no-of-cells-in IS}, we thus obtain $\# \left\{I_{i,7}(S)\colon I_{i,7}(S) \neq \varnothing,\; i\in [1, 6]\right\} = 3$. Indeed, a modest computation confirms that
\[I(S) = \{15,22,29,31,36,38,43,45,47,50,52,54,57,59,61,66,68,73,75,82,89\}\] and
\begin{itemize}
	\item $I_{1,7}(S)= \{15,22,29,36,43,50,57\}$
	\item $I_{2,7}(S)=\varnothing$
	\item $I_{3,7}(S)= \{31,38,45,52,59,66,73\}$
	\item $I_{4,7}(S)=\varnothing$
	\item $I_{5,7}(S)=\{ 47,54,61,68,75,82,89\}$
	\item $I_{6,7}(S)=\varnothing$.
\end{itemize}
\end{example}

Smith proved that the partition $\left\{I_{i,a}(\langle a, b\rangle)\colon I_{i,a}(\langle a, b\rangle) \neq \varnothing,\; i\in [1, a-1]\right\}$ of $I(\langle a, b\rangle)$ is uniform \cite[Theorem 2.7]{smith-tjm22}. Combining this fact with Theorem \ref{th_card-of-iso-gaps-set} and Proposition \ref{pr_no-of-cells-in IS}, we immediately obtain the following proposition. 

\begin{proposition}\label{pr_card-of-cells-in IS}
If $S = \langle a, b\rangle$, then the cardinality of each block of the partition $\left\{I_{i,a}(S)\colon I_{i,a}(S) \neq \varnothing,\; i\in [1, a-1]\right\}$ of $I(S)$ is $|u|$.
\end{proposition}

We illustrate Proposition \ref{pr_card-of-cells-in IS} with the following example.
\begin{example}\label{ex_three-blocks-elements}
Consider the numerical semigroup $S =\langle 7,16 \rangle$. From Example \ref{ex_no-of-blocks}, we know that $|u| = 7$. By using Proposition \ref{pr_card-of-cells-in IS}, we thus obtain that the cardinality of each block of $\left\{I_{i,7}(S)\colon I_{i,7}(S) \neq \varnothing,\; i\in [1, 6]\right\}$ is $|u| = 7$. Indeed, Example \ref{ex_no-of-blocks} says that the cardinality of each block of $\left\{I_{i,7}(S)\colon I_{i,7}(S) \neq \varnothing,\; i\in [1, 6]\right\}$ is $7$.
\end{example}

\section{Isolated Gaps}
We assume throughout this section that $S=\langle a,b\rangle$, and let $h\colonequals \min I(S)$. 

\vspace{0.5mm}
Smith proved that every element of $I(S)$ has the form $h+s$ for some $s\in S$  \cite[Theorem 2.11]{smith-tjm22}. To identify the value of $h$, Smith showed that either $h = 1$ or $\left(h\pmod a, h\pmod b\right)\in \left\{(1, -1), (-1, 1)\right\}$ \cite[Proposition 3.3]{smith-tjm22}.

\vspace{0.5mm}
In this section, we explicitly describe all the isolated gaps of $S$, and find an exact formula for the smallest isolated gap $h$ of $S$.

\vspace{0.5mm}
 For each $i\in [1, a-1]$ such that $I_{i,a}(S)\neq \varnothing$, the smallest element of $I_{i,a}(S)$, written $h_i$, is a \emph{minimal isolated gap modulo $a$} of $I(S)$ \cite[p. 124]{smith-tjm22}. Smith \cite{smith-tjm22} gave the following theorem, which determines all the minimal isolated gaps modulo $a$ of $I(S)$.

\begin{theorem}\cite[Theorem 2.11]{smith-tjm22}\label{th_minimal-isolated-gaps}
Let $n$ be the smallest positive integer such that $h+nb \in a\N$. Then $h, h+b, \cdots, h+(n-1)b$ are all the minimal isolated gaps modulo $a$ of $I(S)$.
\end{theorem}

By combining Proposition \ref{pr_no-of-cells-in IS} with Theorem \ref{th_minimal-isolated-gaps}, we immediately obtain the following proposition.

\begin{proposition}\label{pr_min-iso-gaps}
There are exactly $|v|$ minimal isolated gaps modulo $a$ of $I(S)$, namely, 
$h, h+b, \cdots, h+(|v|-1)b$.
\end{proposition}

\vspace{0.2mm}
To determine all elements of any block of $\left\{I_{i,a}(S)\colon I_{i,a}(S) \neq \varnothing,\; i\in [1, a-1]\right\}$, we require a list of lemmas.

\begin{lemma}\label{lem-Ii(S)-elements-form}
Let $I_{i,a}(S)\neq \varnothing$ for some $i\in [1, a-1]$. Then $x \in I_{i,a}(S)$ if and only if $x = h_i+ka$ for some $k\in \N$. 
\end{lemma}
\begin{proof}[{\bf Proof}]
Assume that $x\in I_{i,a}(S)$. Since $h_i\in I_{i,a}(S)$, we deduce that $x\equiv h_i\pmod a$. Note that $h_i = \min I_{i,a}(S)$, and hence $x = h_i+ka$ for some $k\in \N$.

\vspace{0.2mm}
Conversely, assume that $x = h_i+ka$ for some $k\in \N$. Clearly $ka\equiv 0\pmod a$. Since $h_i\in I_{i,a}(S)$, we find $h_i \equiv i \pmod a$. Thus $x= h_i+ka \equiv i\pmod a$, and hence $x \in I_{i,a}(S)$. 
\end{proof}


\begin{lemma}\label{le_either-in-S-or-IS}
If $x \in I_{i,a}(S)$ for some $i\in [1, a-1]$, then either $x +a \in S$ or $x+a \in I_{i,a}(S)$. 
\end{lemma}
\begin{proof}[{\bf Proof}]
Assume that $x+ a \notin S$. Since $x \in I_{i,a}(S)$, we obtain $x-1, x+1 \in S$. Therefore $(x+a)-1, (x+a)+1 \in S$. Hence $x+a \in I_{i,a}(S)$.
\end{proof}


\begin{lemma}\label{le_smallest-n-x+na-in-bN}
Let $x \in I(S)$. If $n$ is the smallest positive integer such that $x +na \in S$, then $x +na \in b\mathbb{N}$.
\end{lemma}

\begin{proof}[{\bf Proof}]
Assume that $n$ is the smallest positive integer such that $x +na \in S$. Since $x+na \in S$, there exist $\ell,k \in \mathbb{N}$ such that $x+na = \ell a+kb$. We first claim that $\ell < n$. Suppose to the contrary that $\ell \ge n$. Then $x= (\ell-n)a+kb$, and so $s\in S$. This leads to a contradiction, since $x \in I(S)$. 

\vspace{0.2mm}
Now since $kb\in S$ and $x+na = \ell a+kb$, we deduce that $x +(n-\ell)a \in S$. Since $n$ is the smallest positive integer such that $x +na \in S$, it follows that $\ell=0$. Thus $x+na = kb$. Note that $bk \in b\mathbb{N}$, and hence $x+na\in b\mathbb{N}$.
\end{proof}


\begin{lemma}\label{le_elements-of-each-cells}
Let $I_{i,a}(S)\neq \varnothing$ for some $i\in [1, a-1]$. If $n$ is the smallest positive integer such that $h_i +na \in b\mathbb{N}$, then $I_{i,a}(S) = \{ h_i,h_i +a, \cdots, h_i+(n-1)a\}$.
\end{lemma}

\begin{proof}[{\bf Proof}]
Assume that $n$ is the smallest positive integer such that $h_i +na \in b\mathbb{N}$. Then, by Lemma \ref{le_smallest-n-x+na-in-bN} we deduce that $h_i, h_i+a, \cdots, h_i+(n-1)a \notin S$. Therefore $\{h_i, h_i+a, \cdots, h_i+(n-1)a\} \subseteq I_{i,a}(S)$ by Lemma \ref{le_either-in-S-or-IS}.

\vspace{0.2mm}
For the reverse inclusion, let $x\in I_{i,a}(S)$. Then, by Lemma \ref{lem-Ii(S)-elements-form} we obtain $x = h_i+ka$ for some $k\in \N$. 
To prove that $x\in \{h_i, h_i+a, \cdots, h_i+(n-1)a\}$, it suffices to show that $k<n$. Suppose to the contrary that $k\ge n$. Then we can write $x= h_i+na +(k-n)a$. Since $h_i+na \in b\N$ by the hypothesis, it follows that $x\in S$. This leads to a contradiction, since $x\in I_{i,a}(S)$. Hence $k<n$, as desired. 
\end{proof}

By combining Lemma \ref{le_elements-of-each-cells} with Proposition \ref{pr_card-of-cells-in IS}, we immediately obtain the following Proposition.

\begin{proposition}\label{pr_block-isolated-gaps}
If $I_{i,a}(S)\neq \varnothing$ for some $i\in [1, a-1]$, then 
\[I_{i,a}(S) = \{ h_i,h_i +a, \cdots, h_i+(|u|-1)a\}.\]
\end{proposition}


The next theorem determines all the isolated gaps of $S$.

\begin{theorem}\label{th_list-of-isolated-gaps}
The set of isolated gaps of $S$ is given as
\begin{align*}
I(S) &=\{h, h+a, \ldots, h+\left(|u|-1\right)a\}\cup \{h+b, h+b+a, \ldots, h+b+(|u|-1)a\}\\
&\quad \cup \cdots\cup \{h+(|v|-1)b, h+(|v|-1)b+a,\ldots,  h+(|v|-1)b+(|u|-1)a\}.
\end{align*}
\end{theorem}

\begin{proof}[{\bf Proof}]
This follows from Propositions \ref{pr_min-iso-gaps} and \ref{pr_block-isolated-gaps}.
\end{proof}

\vspace{0.5mm}

Theorem \ref{th_list-of-isolated-gaps} suggests that the identifying the value of the smallest isolated gap of $S$ is crucial in calculating all the isolated gaps of $S$. The following proposition gives an exact formula for the smallest isolated gap of $S$ .

 
\begin{proposition}\label{pr_min-isolated-gap}
If $S = \langle a,b \rangle$, then $\min I(S) = F(S)-(|u|-1)a - (|v|-1)b$.
\end{proposition}

\begin{proof}[{\bf Proof}]
Observe that $\max I(S) = h+(|u|-1)a+(|v|-1)b$. From \cite[p. 125]{smith-tjm22}, we know that $F(S) \in I(S)$, and so $\max I(S) = F(S)$. Thus $h+(|u|-1)a+(|v|-1)b = F(S)$, and hence $h= F(S)-(|u|-1)a - (|v|-1)b$.  
\end{proof}


We illustrate Proposition \ref{pr_min-isolated-gap} with the next two examples.

\begin{example}
Consider the numerical semigroup $S= \langle 9, 13 \rangle$. By \cite[Proposition 2.13]{rosales-b09}, we obtain $F(S) = 9\times 13-9-13 = 95$. 
We know from Example \ref{ex-1_card-of-IS} that $|u| = 3$ and $|v| = 2$. By using Proposition \ref{pr_min-isolated-gap}, we thus obtain $\min I(S)= 95- 2\times 9 - 1\times 13 = 64$.  
\end{example}


\begin{example}
Consider the numerical semigroup $S= \langle 7, 11 \rangle$. By \cite[Proposition 2.13]{rosales-b09}, we obtain $F(S) = 7\times 11-7-11 = 59$. Observe that $(-3, 2)$ is a solution of the equation $7x+11y = 1$. Moreover, the solution $(-3, 2)$ is the d.l.s. by Lemma \ref{le_iff-condition-for-dls}. Therefore $|u| = 3$ and $|v| = 2$. By using Proposition \ref{pr_min-isolated-gap}, we thus obtain $\min I(S)= 59- 2\times 7 - 1\times 11 = 34$. Indeed, a modest computation confirms that $I(S) = \{34, 41, 45, 48, 52, 59\}$.
\end{example}

\vspace{0.3mm}
To facilitate further computations on $I(S)$, we construct the $|v|\times |u|$ matrix $L(S)$ whose $(i,j)$-entry is given by $L(S)(i, j) = h+(i-1)b+(j-1)a$. That is,  

\begin{equation*}
	L(S) = 
	\begin{bmatrix}
		h & h+a & \cdots & h+\left(|u|-1\right)a \\
		h+b & h+b+a & \cdots & h+b+(|u|-1)a \\
		\vdots  & \vdots  & \ddots & \vdots  \\
		h+(|v|-1)b & h+(|v|-1)b+a & \cdots & h+(|v|-1)b+(|u|-1)a 
	\end{bmatrix}.
\end{equation*}

\vspace{1.0mm}
For example, consider the numerical semigroup $S= \langle 7, 16 \rangle$. From Example \ref{ex_three-blocks-elements}, we know that $h= 15$, $|u| = 7$, and $|v| = 3$. Thus
\[L(S)=
\begin{bmatrix}
	15 & 22 & 29 & 36 & 43 & 50 & 57\\
	31 & 38 & 45 & 52 & 59 & 66 & 73\\
	47 & 54 & 61 & 68 & 75 & 82 & 89
\end{bmatrix}.\]

By using Proposition \ref{pr_min-isolated-gap}, we can alternatively calculate any $(i,j)$-entry of the matrix $L(S)$ as follows. 
\begin{corollary}
The $(i,j)$-entry of the $|v|\times |u|$ matrix $L(S)$ is given by the rule: 
\[L(S)(i, j) = F(S)-(|u|-j)a- (|v|-i)b.\]
\end{corollary}

\begin{proof}[{\bf Proof}]
By Proposition \ref{pr_min-isolated-gap}, we know that $h = F(S)-(|u|-1)a - (|v|-1)b$. Therefore
\begin{align*}
L(S)(i, j) &= h+(i-1)b+(j-1)a\\
&= F(S)-(|u|-1)a - (|v|-1)b+(i-1)b+(j-1)a\\
&= F(S)-(|u|-j)a- (|v|-i)b.
\end{align*}
\end{proof}


\end{document}